\documentclass[12pt]{article}
\usepackage{a4wide}
\usepackage{amsthm}
\usepackage{amsfonts}
\usepackage{amssymb}
\usepackage{amsmath}
\usepackage{cite}
\usepackage{epsfig}

\newtheorem{theorem}{Theorem}
\newtheorem{lemma}[theorem]{Lemma}
\newtheorem{proposition}[theorem]{Proposition}
\newtheorem{question}{Question}
\newtheorem{conjecture}{Conjecture}

\newcommand{\zloty}{z\l{}oty}

\begin{document}

\title{Optimal-size clique transversals in chordal graphs}
\author{Jacob W.~Cooper\thanks{Department of Mathematics and Statistics, McGill University, Burnside Hall, 805 Sherbrooke West, Montreal, QC, H3A~0B9, Canada. E-mail: {\tt jacob.cooper@mail.mcgill.ca}. Previous affiliation: Department of Computer Science, University of Warwick, Coventry CV4 7AL, UK.}\and
        Andrzej Grzesik\thanks{Department of Computer Science, University of Warwick, Coventry CV4 7AL, UK, and Faculty of Mathematics and Computer Science, Jagiellonian University, \L ojasiewicza 6, 30-348 Krak\'ow, Poland. E-mail: {\tt Andrzej.Grzesik@tcs.uj.edu.pl}. Previous affiliation: Institute of Informatics, University of Warsaw, ul. Banacha 2, 02-097 Warsaw, Poland. E-mail: {\tt grzesik@mimuw.edu.pl}.} \and
        Daniel Kr\'al'\thanks{Mathematics Institute, DIMAP and Department of Computer Science, University of Warwick, Coventry CV4 7AL, UK. E-mail: {\tt d.kral@warwick.ac.uk}.}}
\date{}	
\maketitle

\begin{abstract}
The following question was raised by Tuza in 1990 and Erd\H os et al.~in 1992:
if every edge of an $n$-vertex chordal graph $G$ is contained in a clique of size at least four,
does $G$ have a clique transversal, i.e.,~a set of vertices meeting all non-trivial maximal cliques, of size at most $n/4$?
We prove that every such graph $G$ has a clique transversal of size at most $2(n-1)/7$ if $n\ge 5$,
which is the best possible bound.
\end{abstract}

\section{Introduction}

We investigate a problem posed by Erd\H os et al.~\cite[Problem 4]{bib-erdos92+} and Tuza~\cite[Problem 1]{bib-tuza90}
on clique transversals in chordal graphs.
To state the problem, we require the following definitions. 
A graph is {\em chordal} if it contains no induced cycle of length four or more.
We define a {\em clique} to be a complete subgraph of a graph (here, we deviate from the standard definition) and 
call a clique {\em maximal} if it is an inclusion-wise maximal complete subgraph. 
A {\em $k$-clique} is a clique containing $k$ vertices; moreover we call a clique non-trivial if it is a $k$-clique for some $k\ge 2$.
A {\em clique transversal} of a graph $G$ is a set $U$ of vertices such that
every non-trivial maximal clique of $G$ contains a vertex from $U$.
Finally, a chordal graph $G$ is {\em $k$-chordal} if each edge of $G$ is contained in a $k$-clique;
however, a $k$-chordal graph can also contain maximal cliques with less than $k$ vertices.

Every $2$-chordal $n$-vertex graph has a clique transversal of size at most $n/2$~\cite{bib-aigner+}, and
every $3$-chordal $n$-vertex graph has a clique transversal of size at most $n/3$~\cite{bib-tuza90}.
Motivated by these results,
Erd\H os et al.~\cite[Problem 4]{bib-erdos92+} and Tuza~\cite[Problem 1]{bib-tuza90} posed the following.
\begin{question}
\label{quest}
Does every $4$-chordal graph with $n$ vertices have a clique transversal of size at most $n/4$?
\end{question}
\noindent Flotow~\cite{bib-flotow92} constructed counterexamples for some small values of $n$ and 
Andreae and Flotow~\cite{bib-andreae96+}
constructed arbitrarily large $n$-vertex $4$-chordal graphs admitting no clique transversal with fewer than $2n/7-O(1)$ vertices, 
conjecturing this to be tight up to a constant factor.
\begin{conjecture}
\label{conj}
Every $4$-chordal graph with $n$ vertices has a clique transversal of size at most $2n/7$.
\end{conjecture}
\noindent This conjecture also appears as~\cite[Problem 77]{bib-tuza01}.
Andreae~\cite{bib-andreae98} proved this conjecture when restricting to $4$-chordal graphs with every maximal clique 
containing at most four vertices.

In this paper, we prove Conjecture~\ref{conj} and
determine the optimal function of $n$ for which Question~\ref{quest} holds true, that is we prove the following.
\begin{theorem}
\label{thm-main}
Every $4$-chordal graph $G$ with $n\ge 5$ vertices
has a clique transversal of size at most $\lfloor 2(n-1)/7\rfloor$.
\end{theorem}
\noindent As shown in Proposition~\ref{prop-lower},
the bound proven in Theorem~\ref{thm-main} is best possible for every $n\ge 5$.

\section{Tree-decompositions of chordal graphs}

It is well-known that every chordal graph is an intersection graph of subtrees of a tree~\cite{bib-gavril74}.
Such intersection representations lead to tree-decompositions, which we now define;
we remark that the notion that we use differs from tree-decompositions related to graph tree-width.
A {\em tree-decomposition} of a graph $G$ is a tree $T$ such that
\begin{itemize}
\item each node $u$ of $T$ is {\em associated} with a subset $V_u$ of vertices of $G$,
\item two vertices $v$ and $v'$ of $G$ are joined by an edge if and only if there exists a node $u$ such that $\{v,v'\}\subseteq V_u$, and
\item for every vertex $v$ of $G$, the nodes $u$ with $v\in V_u$ induce a subtree of $T$.
\end{itemize}
A graph $G$ admits a tree-decomposition $T$ if and only if $G$ is chordal.
We will always refer to the vertices of a tree-decomposition as nodes in order to distinguish clearly
between the vertices of graphs and tree-decompositions.
In addition, we will generally use the letter $u$ (with various subscripts and superscripts) for nodes,
$U$ for sets of nodes, $v$ for vertices, and $V$ for sets of vertices.

The vertices associated with a node $u$ of a tree-decomposition of a graph $G$
form a clique in $G$; we will say that $u$ {\em corresponds} to this clique.
The Helly property of subtrees of a tree implies that
every maximal clique corresponds to a node of a tree-decomposition.
However, not all nodes of a tree-decomposition need correspond to a maximal clique.
We now state the following folklore lemma that asserts it is possible to modify a tree-decomposition of a chordal graph such that
each node corresponds to a maximal clique.

\begin{lemma}
\label{lm-maximal}
Every chordal graph $G$ has a tree-decomposition $T$ such that
each node of $T$ corresponds to a maximal non-trivial clique of $G$, and
all the nodes of $T$ are associated with different subsets of vertices of $G$.
\end{lemma}

\begin{proof}
Let $T$ be a tree-decomposition of $G$ with the minimum number of nodes.
The minimality of $T$ implies that no node is associated with a single vertex,
i.e., all cliques associated with nodes of $T$ are non-trivial.
Let $u$ be a node of $T$ and let $C$ be the clique corresponding to $u$.
We will show that $C$ is a maximal clique.
If $u$ had a neighbor $u'$ such that every vertex associated with $u$ were also associated with $u'$,
we could contract the edge $uu'$ and obtain a smaller tree-decomposition of $G$,
which is impossible by the choice of $T$.
Hence, for every neighbor $u'$, there exists a vertex $v$ associated with $u$ but not with $u'$.

Suppose that $C$ is not a maximal clique, i.e., there exists a clique $C'$ of $G$ containing $C$.
Let $u''$ be the node that corresponds to $C'$ and let $u'$ be the neighbor of $u$ on the path from $u$ to $u''$.
Furthermore, let $v$ be a vertex associated with $u$ but not with $u'$.
Since the nodes associated with $v$ induce a subtree, the vertex $v$ is not associated with the node $u''$.
Consequently, $C'$ does not contain $v$. We conclude that the clique $C$ is maximal.
\end{proof}

In the proof of Theorem~\ref{thm-main}, we require a particular type of tree-decomposition, which we now define.
A tree-decomposition of a $4$-chordal graph is {\em nice} if it is rooted and satisfies the following:
\begin{itemize}
\item no two adjacent nodes correspond to the same clique,
\item if a node $u$ corresponds to a $k$-clique $C$, then $k\ge 3$ and if $k=3$, then $C$ is maximal,
\item if a node $u$ corresponds to a $k$-clique, $k\ge 5$ and $u$ is not the root, then
      $k-1$ vertices associated with $u$ are also associated with the parent of $u$, and
\item if $G$ contains a maximal $3$-clique, then the root node corresponds to a maximal $3$-clique.
\end{itemize}
Observe that each leaf of a nice tree-decomposition of a $4$-chordal graph
corresponds to a $k$-clique with $k\ge 4$.

We now show that every $4$-chordal graph admits a nice tree-decomposition,
even with an additional restriction on the clique corresponding to its root.

\begin{proposition}
\label{prop-nice}
Let $G$ be a $4$-chordal graph and let $C$ be a maximal $3$-clique of $G$, if it exists;
otherwise let $C$ be a clique of $G$ of order at least four.
Then, $G$ admits a nice tree-decomposition such that its root corresponds to $C$.
\end{proposition}

\begin{proof}
Let $T$ be a tree-decomposition of $G$ such that each node of $T$ corresponds to a maximal non-trivial clique,
which exists by Lemma~\ref{lm-maximal}.
We construct a nice tree-decomposition of $G$ from $T$.

We first introduce a new node $r$ and associate it with the vertices of $C$.
We join $r$ to a node of $T$ that corresponds to a clique containing $C$ (if $C$ is maximal,
then this node corresponds to $C$ itself), and
root the tree-decomposition at $r$.
Observe that the second and fourth properties from the definition of a nice tree-decomposition hold in this modified tree-decomposition.

While the modified tree-decomposition contains a node $u$ associated with vertices $v_1,\ldots,v_k$, $k\ge 5$, such that
two of these vertices, say $v_{k-1}$ and $v_k$, are not associated with the parent $u'$ of $u$, we proceed as follows:
we introduce a new node $u''$ to be both a child of $u'$ and the parent of $u$, 
associating $u''$ with $v_1,\ldots,v_{k-1}$.
This process terminates since the sum of the squares of the sizes of the symmetric difference of vertex sets associated
with adjacent nodes in the tree-decomposition decreases in each step. 

This tree-decomposition now satisfies the second, third and fourth properties from the definition of a nice tree-decomposition.
If the tree-decomposition contains two adjacent nodes corresponding to the same clique,
we contract the edge joining them. After this process terminates, we have a nice tree-decomposition of $G$
rooted at a node corresponding to the clique $C$.
\end{proof}

\section{Clique transversals in chordal graphs}

To prove our main result, we design an algorithm that constructs a small clique transversal of a given $4$-chordal graph $G$, 
whose input is $G$ together with any one of its nice tree-decompositions.
During the course of the algorithm, nodes of the input nice tree-decomposition will be removed sequentially, 
whilst simultaneously in $G$, selected pairs and triples of vertices will become temporarily distinguished, determining which 
vertices are to be removed or colored red with every iteration. 
Distinguished pairs and triples will be collectively referred to as distinguished tuples, and 
it will hold that:
\begin{itemize}
\item a distinguished tuple will never contain a red vertex,
\item the vertices of each distinguished tuple will induce a complete subgraph, and
\item at least one vertex of a distinguished tuple will eventually become red.
\end{itemize}
The red vertices will form the sought clique transversal.
We will refer to an algorithm of this kind as a {\em clique transversal algorithm}.

To bound the size of the constructed clique transversal, we apply a double counting argument.
Initially, each vertex will be assigned two \zloty{}s and
each node of the input nice tree-decomposition corresponding to a maximal $3$-clique will be assigned one \zloty{}.
As each vertex of the chordal graph is removed or colored red,
its assigned \zloty{}s will be reassigned in accordance with the algorithm.
Whenever a vertex is colored red, seven \zloty{}s will be removed from the graph---these seven \zloty{}s will be referred to as {\em paid} for coloring a vertex red.
It may also occur that some \zloty{}s are removed from the graph without 
creating a red vertex---these \zloty{}s will be referred to as {\em saved}.
In addition, each distinguished pair will always be assigned three \zloty{}s and
each distinguished triple two \zloty{}s at the instant of its creation.
When any distinguished tuple ceases to exist, its \zloty{}s will also be reassigned.

To aid the accessibility of our arguments,
we will sequentially design three algorithms with the properties above.
To establish a further notion, a clique transversal algorithm will be denoted {\em cash-balanced}
if it abides by the rules presented above for reassigning \zloty{}s.

\begin{proposition}
\label{prop-zlotys}
Let $G$ be a $4$-chordal graph and let $X$ be the clique transversal produced by a clique transversal algorithm that is cash-balanced.
If $n$ is the number of vertices of $G$, $t$ is the number of maximal $3$-cliques of $G$ and $s$ is the number of \zloty{}s saved by the algorithm, then
$$|X|=\frac{2n+t-s}{7}\;\mbox{.}$$
\end{proposition}

\begin{proof}
Initially, $2n$ \zloty{}s are assigned amongst the vertices of $G$ and
$t$ \zloty{}s are assigned to maximal $3$-cliques. 
Since $s$ \zloty{}s are saved during the course of the algorithm, it stands that $2n+t-s$ \zloty{}s have been paid for red vertices.
Since exactly seven \zloty{}s are paid for each red vertex,
the bound on the size of $X$ follows.
\end{proof}

\subsection{Basic algorithm}

We begin by presenting a clique-transversal algorithm that may fail to save any \zloty{}s.
We will then modify this algorithm for the case of $4$-chordal graphs with maximal $3$-cliques and
finally the case of $4$-chordal graphs with no maximal $3$-cliques.
The algorithm that we design in this subsection will be referred to as the {\em basic algorithm}.

The basic algorithm processes the nodes of an input nice tree-decomposition of a $4$-chordal graph from the leaves towards the root using the rules we now describe.
Each time a node is processed, it is removed from the tree-decomposition;
moreover, the vertices of the graph associated with only that node are also removed.
Note that the tree-decomposition remains nice after this operation.

Let $u$ be a leaf node of the tree-decomposition,
let $U$ be the set of vertices associated with $u$, and
let $U'\subseteq U$ be the set of vertices associated with $u$ but not with its parent.
We analyze four cases concerning the possible sizes of $U$ and $U'$;
we start with the case corresponding to the most common scenario.

{\bf Case 1: The size of $U$ is at least four and the size of $U'$ is one.}

Let $v$ be the single vertex contained in $U'$.
The algorithm follows the first of the following rules that applies 
and then removes both 
the node $u$ from the tree-decomposition and the vertex $v$ from the graph.
We remark that at least one \zloty{} is saved when one of Rules G1--G5 is applied,
so we will refer to these five rules as {\em good} rules.
Similarly, rules B1--B5 will be referred to as {\em bad} rules (since no \zloty{}s are saved).

\begin{description}
\item[Rule B1.] If the vertex $v$ is red, the algorithm performs no additional steps.
\item[Rule G1.] If the vertex $v$ is contained in at least three distinguished tuples, it becomes red.
                The tuples containing $v$ cease to exist and their \zloty{}s are removed.
		Each of the tuples has at least two \zloty{}s and $v$ has an additional two \zloty{}s itself.
                Hence, seven \zloty{}s are paid for coloring $v$ red and at least one \zloty{} is saved.
\item[Rule G2.] If the vertex $v$ is contained in at least two distinguished pairs, it becomes red.
                The pairs containing $v$ cease to exist and their \zloty{}s are removed.
                As in the case of Rule G1, at least one \zloty{} is saved.
\item[Rule B2.] If the vertex $v$ is contained in a distinguished pair and a distinguished triple,
                it becomes red and the two tuples cease to exist. No \zloty{}s are saved.
\item[Rule B3.] If the vertex $v$ is contained in two distinguished triples,
                the remaining two vertices of each triple will become a distinguished pair. 
                Each pair keeps the \zloty{}s of the original triple it was contained in and
                is reassigned one additional \zloty{} from $v$.
\item[Rule G3.] If the vertex $v$ is contained in a distinguished triple,
                the remaining two vertices of the triple become a distinguished pair.
                This pair keeps the two \zloty{}s of the original triple and
                gets one \zloty{} from $v$. The other \zloty{} of $v$ is saved.
\item[Rule G4.] If the vertex $v$ is contained in a distinguished pair and
                the other vertex, call it $v'$, of that pair is contained in another distinguished tuple,
                the vertex $v'$ becomes red. The tuples containing $v'$ cease to exist.
		Since we remove at least nine \zloty{}s (the four \zloty{}s of $v$ and $v'$ and
		at least five \zloty{}s of the tuples),
                at least two \zloty{}s are saved.
\item[Rule B4.] If the vertex $v$ is contained in a distinguished pair,
                the other vertex of the pair becomes red and the pair ceases to exist.
\item[Rule G5.] If the vertex $v$ is not in a tuple and one of the vertices of $U\setminus\{v\}$ is red,
                the algorithm also performs no additional steps;
                the two \zloty{}s of $v$ are saved.
\item[Rule B5.] If the vertex $v$ is not in a tuple and no vertex of $U\setminus\{v\}$ is red, 
				an arbitrary three vertices of $U\setminus\{v\}$ form a distinguished triple and
                this triple is reassigned two \zloty{}s from $v$.
\end{description}

{\bf Case 2: The size of $U$ is at least four and the size of $U'$ is greater than one.}

We apply some of the rules from Case 1.
As long as $U'$ contains a non-red vertex that is not contained in any distinguished tuple (and that has not been processed),
we process it, i.e., we use Rule B5 or G5.
We then use the first applicable rule to process each remaining vertex in $U'$ in an arbitrary 
order, followed by removing the node $u$ from the tree-decomposition.

{\bf Case 3: The size of $U$ is three and the size of $U'$ is one.}

Let $v$ be the single vertex contained in $U'$.
The algorithm follows the first of the following rules that applies and
then removes the node $u$ from the tree-decomposition and the vertex $v$ from the graph.
\begin{description}
\item[Rule T1.] If the vertex $v$ is red, the algorithm performs no additional steps.
\item[Rule T2.] If the vertex $v$ is contained in at least two distinguished tuples,
                it becomes red and all the tuples containing $v$ cease to exist. 
		Coloring $v$ red is paid using the two \zloty{}s of $v$,
		at least four \zloty{}s from the tuples containing $v$ and
		the one \zloty{} of node $u$.
\item[Rule T3.] If the vertex $v$ is contained in exactly one distinguished tuple,
                an arbitrary vertex $v'$ of the tuple different from $v$ becomes red and
                all distinguished tuples containing $v'$ cease to exist.
                Note that at least seven \zloty{}s are removed with this rule:
                the two \zloty{}s of the vertex $v$,
                the two \zloty{}s of the new red vertex,
                at least two \zloty{}s from the ceased tuples, and
                the one \zloty{} of the node $u$.
\item[Rule T4.] If $U$ contains a red vertex, the algorithm performs no additional steps;
                one \zloty{} assigned to the node $u$ is saved.
\item[Rule T5.] The other two vertices of $U$ become a distinguished pair,
                which is assigned the two \zloty{}s of $v$ and the one \zloty{} of node $u$.
\end{description}

{\bf Case 4: The size of $U$ is three and the size of $U'$ is greater than one.}

Let $v$ be an arbitrary vertex of $U'$.
We start with following the first applicable rule among the Rules T1--T5 with respect to $v$;
note that at least one of the vertices of $U$ is now red.
If the two vertices of $U$ different from $v$ do not form a distinguished pair,
we next remove the remaining vertices of $U'$ from the graph and also remove the node $u$ from the tree-decomposition.

We now assume that the two vertices of $U$ different from $v$ form a distinguished pair.
Let $v'$ be the vertex of $U\setminus U'$ if it exists;
otherwise, it holds $U=U'$ and let $v'$ be an arbitrary vertex of $U$ different from $v$.
The vertex $v'$ is now colored red and all distinguished pairs containing $v'$ cease to exist.
The remaining vertices of $U'$ and the node $u$ are then removed. 
Note that at least seven \zloty{}s are removed: the four \zloty{}s of the two vertices $v$, $v'$ and
three \zloty{}s from the distinguished pair.

This concludes the description of the basic algorithm.
Its description yields that the basic algorithm is a cash-balanced clique transversal algorithm.
Note that if we apply the basic algorithm to a nice tree-decomposition of a $4$-chordal graph $G$ with no maximal $3$-cliques,
we get a clique transversal of $G$ of size at most $2|G|/7$ by Proposition~\ref{prop-zlotys}.

\subsection{Chordal graphs with maximal $3$-cliques}

A {\em branch} is a rooted subtree of a nice tree-decomposition $T$ of a $4$-chordal graph such that:
\begin{itemize}
\item its root $r$ corresponds to a $4$-clique,
\item the parent of $r$ corresponds to a clique that shares two vertices with $r$,
\item the branch contains all descendants of $r$ in $T$, and
\item no descendant of $r$ corresponds to a $3$-clique.
\end{itemize}
The following proposition easily follows from the definition of a nice tree-decomposition.

\begin{proposition}
\label{prop-branch}
Let $G$ be a $4$-chordal graph with $t\ge 1$ maximal $3$-cliques.
Every nice tree-decompo\-si\-tion of $G$ has at least $t+2$ branches.
\end{proposition}

\begin{proof}
Let $T$ be a nice tree-decomposition of $G$.
Suppose that $u$ is a node of $T$ corresponding to a $3$-clique formed by vertices $v_1$, $v_2$ and $v_3$.
Since $G$ is a $4$-chordal graph, every pair of the vertices $v_1$, $v_2$ and $v_3$ is contained in a $4$-clique.
In particular, for every such pair, there exists a node adjacent to $u$ that contains both vertices of the pair.
Let $u_1$, $u_2$ and $u_3$ be these nodes for the three different pairs of the vertices $v_1$, $v_2$ and $v_3$.
Since the $3$-clique corresponding to $u$ is maximal, the nodes $u_1$, $u_2$ and $u_3$ are different.
Observe that if $u_i$, $i\in\{1,2,3\}$, is a child of $u$, then the subtree rooted at $u_i$ is a branch
unless it contains a node corresponding to a $3$-clique.

Let $T'$ be a rooted tree obtained from $T$ as follows:
the nodes of $T'$ are the nodes of $T$ corresponding to $3$-cliques;
the root of $T'$ is the root of $T$; and
a node $u$ is the parent of a node $u'$ if the path from $u$ to $u'$ contains no node corresponding to a $3$-clique.
Hence, every leaf of $T'$ will have at least two children in $T$ such that the subtrees rooted at them form a branch, and
every node with exactly one child in $T'$ has at least one child in $T$ such that the subtree rooted at it is a branch.
Moreover, if the root of $T'$ has $d$ children in $T'$ and $d<3$,
it has at least $3-d$ children in $T$ such that the subtrees rooted at them form a branch.
We conclude that if the degree of a node of $T'$ is $d$,
this node in $T$ has at least $3-d$ children such that the subtrees rooted at them form a branch.
Since $T'$ has $t$ nodes and the sum of their degrees is $2(t-1)$, we derive that $T$ must have at least $3t-2(t-1)=t+2$ branches.
\end{proof}

We now modify the basic algorithm to save at least one \zloty{} while processing each branch of an input nice tree-decomposition.

\begin{lemma}
\label{lm-branch}
There exists a cash-preserving clique transversal algorithm that
saves at least one \zloty{} when processing the nodes of each branch of an input nice tree-decomposition.
\end{lemma}

\begin{proof}
Consider a branch rooted at a node $r$.
Observe that we can freely choose the order in which the nodes of the branch are processed provided 
that all the descendants of each node are processed before the node itself.
In particular, if we find an order such that at least one of the good rules is used,
we save one \zloty{} as desired.

Let $u$ be an arbitrary leaf node of the branch and
let $u_0=u,u_1,\ldots,u_\ell=r$ be the path from $u$ to $r$ in $T$.
Furthermore, let $\alpha$ and $\beta$ be the two vertices of the clique corresponding to $r$ that
are also contained in the clique corresponding to the parent of $r$.
The node $u$ will be the first node of the branch to be processed.
Suppose first that the clique corresponding to $u$ has two vertices that are not contained in the clique corresponding to $u_1$.
This implies that $u$ corresponds to a $4$-clique.
Also note that the two such vertices are associated with the node $u$ only,
in particular, they cannot be contained in a distinguished tuple.
Hence, the first two rules that apply are either Rules B5 and G3 or Rule G5 twice.
In both cases, a good rule is applied and hence at least one \zloty{} is saved.
The remaining nodes of the branch can be processed in an arbitrary order.

If the clique corresponding to $u$ has exactly one vertex not contained in the clique corresponding to $u_1$,
then either Rule B5 or Rule G5 is the first rule to apply.
In the latter case, a good rule is applied, so we need only analyze the former.
We process all nodes of the branch in an arbitrary order, but stipulate that the nodes $u_1,\ldots,u_{\ell}$ are to be processed last.
We can assume that no good rules are applied while the branch is processed---otherwise, one \zloty{} is saved.
In particular, the distinguished triple created by Rule B5 during the removal of the node $u$
cannot cease to exist before processing the node $u_1$ (without applying Rule G4).

Let $k$ be the largest index such that a subset of the vertices of the clique corresponding to $u_k$ form a distinguished triple 
prior to processing the node $u_k$, and
let $V$ be the set of vertices contained in the cliques corresponding to the nodes $u_k,\ldots,u_\ell$.
By the choice of $k$, either Rule B2 or B3 applies to $u_k$; let $v$ be the vertex that is removed at this step.
Observe that the choice of $k$ implies that, after $v$ is removed,
only Rules B1 and B4 can apply while processing $u_k,\ldots,u_\ell$,
i.e., the distinguished pairs induce a matching (without multiple identical pairs) on non-red vertices of $V$
with the possible exception of $\alpha$, $\beta$ or both, which may be unmatched.

We modify the basic algorithm to remove the vertex $v$ more economically depending on which of 
Rules B2 or B3 removed it, saving at least one \zloty{} as required.

{\bf Case 1: The vertex $v$ was removed by Rule B2.}

Let $v_1$ and $v'_1$ be the two vertices forming a distinguished triple with $v$ and
let $v_2$ be the vertex forming a distinguished pair with $v$.
If $v_2=v_1$ or $v_2=v'_1$, we color $v_2$ red and remove $v$, including the tuples containing it.
In this way, two \zloty{}s are saved.
Hence, we can assume that the vertices $v_1$, $v'_1$ and $v_2$ are mutually distinct.
Note that at least one of these three vertices must be distinct from $\alpha$ and $\beta$.

If $v_2$ is contained in a distinguished pair with another vertex, we proceed as follows.
We color $v_2$ red and remove $v$.
We also remove the distinguished pairs containing $v_2$ and the distinguished triple $\{v,v_1,v'_1\}$, 
before creating a new distinguished pair $\{v_1,v'_1\}$.
This procedure removes $12$ \zloty{}s (two from each of $v$ and $v_2$, three from each distinguished pair 
containing $v_2$ and two from the distinguished triple), which we reassign as follows:
seven \zloty{}s are paid for coloring $v_2$ red,
three \zloty{}s are reassigned to the new distinguished pair, and
two \zloty{}s are saved.

If $v_2$ is not contained in a distinguished pair with another vertex, $v_2$ is either $\alpha$ or $\beta$. 
By symmetry, we can assume that $v_2=\alpha$.
Consequently, at least one of the vertices $v_1$ and $v'_1$ is distinct from $\beta$.
We can assume $v_1\not=\beta$ by symmetry.
Let $v''_1$ be the vertex contained in a distinguished pair with $v_1$.
If $v''_1=\beta$, we swap the roles of $v_1$ and $v'_1$.
In particular, we can assume that neither $v_1$ nor $v''_1$ is $\beta$.
We can now proceed as follows:
the vertices $v_1$ and $v_2=\alpha$ are colored red
while the vertex $v$ and all distinguished tuples containing $v_1$ or $v_2$ are removed.
A total of $14$ \zloty{}s are removed (two from each of the vertices $v$, $v_1$ and $v_2$; five from the 
distinguished tuples containing $v$; and three from the distinguished pair containing $v_1$) 
as payment for coloring the vertices $v_1$ and $v_2$ red.
We run the basic algorithm and observe that the vertex $v''_1$ is removed by Rule G5 (the vertex $\alpha$ is contained
in all the cliques corresponding to the nodes $u_k,\ldots,u_{\ell}$), which results in saving two \zloty{}s.

{\bf Case 2: The vertex $v$ was removed by Rule B3.}

If both distinguished triples containing $v$ also contain another vertex, say $v'$,
we color $v'$ red and remove the vertex $v$, including the two triples containing it.
In this way, we have removed eight \zloty{}s, of which seven are paid for coloring $v'$ red and
one is saved. Hence, we can assume that $v$ is the only vertex contained
in the intersection of the two distinguished triples.

We proceed by coloring the vertex $v$ red and removing the two triples containing it.
Unfortunately, we are only able to pay six out of the seven \zloty{}s requisite to color $v$ red.
We argue that we can always remove two \zloty{}s while processing the nodes $u_{k+1},\ldots,u_\ell$ without creating a red vertex;
one of which will pay off the one \zloty{} debt for coloring $v$ red and 
the other will constitute the required saving in the branch.

Let $x_1,\ldots,x_4$ be the four other vertices contained in the two distinguished triples that contained $v$.
Let $k'$ be the largest index such that the clique corresponding to $u_{k'}$ contains all four of the vertices $x_1,\ldots,x_4$.
By symmetry, we can assume that $x_1$ is not contained in the clique corresponding to the parent of $u_{k'}$.
Note that still Rules B1 and B4 apply solely while processing the nodes $u_{k+1},\ldots,u_{k'-1}$.
In particular, no new distinguished tuples are created and none of the vertices $x_1,\ldots,x_4$ become red.
If the clique corresponding to the node $u_{k'}$ contains a red vertex, we simply apply Rule G5 to $x_1$.
If the clique of $u_{k'}$ contains no vertices in addition to $x_1,\ldots,x_4$,
we may still remove $x_1$ and its two \zloty{}s
since the clique is not maximal (it is a subset of the clique corresponding to $u_k$).
Hence, we can assume that the clique of $u_{k'}$ contains a non-red vertex $y$.

We first assume that $y=\alpha$ (note that the case $y=\beta$ is symmetric).
If no vertex from $x_2,\ldots,x_4$ is $\beta$,
we color $y$ red and observe that all of the vertices $x_1,\ldots,x_4$ are eventually removed by Rule G5.
Thus we have removed $10$ \zloty{}s: seven \zloty{}s are paid for coloring $y$ red,
one \zloty{} pays off the single \zloty{} debt for coloring $v$ red, and two \zloty{}s are saved.
If one of the vertices $x_2,\ldots,x_4$, say $x_4$, is $\beta$,
we create a distinguished pair formed from $y=\alpha$ and $x_4=\beta$, and
when processing the vertices $x_1,x_2,x_3$, simply remove them from the graph (the cliques corresponding to the nodes
containing $x_1,x_2,x_3$ also contain both $\alpha$ and $\beta$, so they will eventually contain a red vertex).
In this way, a total of $6$ \zloty{}s are removed:
one pays off the \zloty{} debt for coloring $v$ red,
three are assigned to the new distinguished pair, and
two \zloty{}s are saved.

If the vertex $y$ is neither $\alpha$ nor $\beta$, it must be contained in a distinguished pair with another vertex, say $y'$.
When processing $u_{k'}$,
we color $y$ red, which is paid for with the three \zloty{}s assigned to the distinguished pair containing $y$ and four 
\zloty{}s shared between the vertices $x_1$ and $y$.
If any of the vertices $x_2,x_3,x_4$ is removed before $y$, we must have applied Rule G5.
The two \zloty{}s we save from this are used to pay off the one \zloty{} debt for coloring $v$ red and make the required saving 
in the branch.
Otherwise, $y$ is removed before the vertices $x_2,x_3,x_4$, 
which implies there is a clique containing each of the vertices $x_2,x_3,x_4,y,y'$.
After $y$ is removed,
the case of the vertices $x_2,x_3,x_4,y'$ is completely analogous to the original case of the vertices $x_1,\ldots,x_4$, and
we proceed in the very same way, i.e., we repeat the steps presented in this and the preceding two paragraphs.
Since the process must eventually terminate, we find one \zloty{} to pay off the debt and also save one \zloty{}.
\end{proof}

Propositions~\ref{prop-zlotys} and~\ref{prop-branch} and Lemma~\ref{lm-branch} yield the following.

\begin{theorem}
\label{thm-main-A}
Every $4$-chordal graph $G$ containing a maximal $3$-clique
has a clique transversal with at most $2(|G|-1)/7$ vertices.
\end{theorem}

\subsection{Chordal graphs with no maximal $3$-cliques}

As identified earlier,
the basic algorithm together with Proposition~\ref{prop-zlotys} gives, in this case, a clique transversal of $G$ 
of size at most $2|G|/7$, which is greater than our proposed bound of $2(|G|-1)/7$.
To attain this improvement, we will modify the basic algorithm depending on the structure of $G$.
In general, the algorithm will follow the steps of the basic algorithm except for a small number of the initial steps.

\begin{theorem}
\label{thm-main-B}
Every $4$-chordal graph $G$ that contains no maximal $3$-clique
has a clique transversal with at most $2(|G|-1)/7$ vertices unless $|G|=4$.
\end{theorem}

\begin{proof}
Let $G$ be a counterexample to the statement of the theorem 
containing the minimum number of vertices and, subject to this, the minimum number of edges.
Clearly, $G$ is connected, in particular, $G$ has no isolated vertices.
Fix a rooted tree-decomposition $T$ of $G$ with the maximum number of leaves such that
every node corresponds to a maximal clique of $G$ and
no two nodes correspond to the same clique (such a tree-decomposition exists by Lemma~\ref{lm-maximal}).
In particular, no node of $T$ corresponds to a $3$-clique.
If $G$ has at most two maximal cliques, then it has a clique transversal of size one.
So, we can assume that $G$ has at least three maximal cliques, which implies 
that the tree-decomposition $T$ has at least three nodes.
In particular, the root of $T$ has at least two children.

We now show that we may assume each leaf node of $T$ corresponds to a $4$-clique.
Suppose that $T$ has a leaf node $u'$ that corresponds to a $k$-clique with $k\ge 5$.
Let $V'$ be the set of vertices of $G$ that are associated with $u'$, and 
let $V$ be the set of vertices associated with the parent of $u'$ in $T$.
If $|V'\setminus V|\ge 2$, then let $v$ be any vertex in $V'\setminus V$ and
let $G'$ be the graph obtained from $G$ by removing the vertex $v$.
If $|V'\setminus V|=1$, then let $v$ be the unique vertex in $V'\setminus V$,
let $v'$ be any vertex of $V'\cap V$, and
let $G'$ be the graph obtained from $G$ by removing the edge $vv'$.
Observe that in both cases, $G'$ is a $4$-chordal graph and
any clique transversal of $G'$ is also a clique-transversal of $G$.
However, this contradicts the choice of $G$ as a minimal counterexample.
We conclude that every leaf node of $T$ corresponds to a $4$-clique.

Suppose that $T$ has a node $u$ adjacent to a leaf $u'$ such that
at least two vertices associated with $u'$ are not associated with $u$.
Let $V$ and $V'$ be the sets of vertices of $G$ associated with nodes $u$ and $u'$ respectively, and
let $k=|V'\setminus V|$. Note that $k\in\{2,3\}$ since $G$ is connected.
Consider a nice tree-decomposition $T'$ such that its root corresponds to the clique induced by $V'$;
such a tree-decomposition $T'$ exists by Proposition~\ref{prop-nice}.
Note that the root of $T'$ has a single child, which corresponds to a $4$-clique, and
exactly the $k$ vertices of $V'\setminus V$ are not associated with its only child.

If $k=2$, the subtree rooted at the only child of the root of $T'$ is a branch and
we process its nodes following Lemma~\ref{lm-branch}, which results in saving one \zloty{}.
If one of the nodes $V'$ is red, then we can save at least $4$ \zloty{}s and
conclude that $G$ has a clique transversal of size at most $(2|G|-5)/7$.
If no node of $V'$ is red, we color one of the vertices of $V\cap V'$ red, making a saving of 
one \zloty{} from the eight \zloty{}s assigned to the vertices of $V'$ (additional \zloty{}s
are saved if the two vertices of $V\cap V'$ form a distinguished pair), and
conclude that $G$ has a clique transversal of size at most $(2|G|-2)/7$.

If $k=3$, then exactly one vertex, say $v$, of $V'$ is also contained in $V$.
Let $v'$ be another vertex associated with the child of the root of $T'$.
Modify $T'$ to $T''$ by associating the vertex $v$ with the root of $T'$;
the subtree of $T''$ rooted at the single child of the root is now a branch,
which shares the vertices $v$ and $v'$ with the rest of the graph.
Using Lemma~\ref{lm-branch}, we process this branch and save one \zloty{}.
The only vertices remaining after processing the branch are those contained in $V'\cup\{v'\}$.
Observe that none of the vertices of $V'\setminus\{v\}$ are red or contained in a distinguished tuple;
in particular, each of them still has two \zloty{}s.
If $v$ is red, then we remove all vertices contained in $V'\cup\{v'\}$ from the graph;
in this way, we save at least six \zloty{}s, which are assigned to the vertices of $V'\setminus\{v\}$.
If $v$ is not red, we color $v$ red and remove all vertices contained in $V'\cup\{v'\}$,
including the distinguished pairs $\{v,v'\}$ if they exist;
we save one \zloty{} out of eight \zloty{}s assigned among the vertices of $V'$.
In total, we save at least two \zloty{}s in both cases and 
conclude that $G$ has a clique transversal of size at most $2(|G|-1)/7$.

From now on, we can assume that all but one vertex associated with each leaf node of $T$
are also associated with its parent in $T$.

We next show that the tree $T$ contains no node $u'$ with a single child and this child is a leaf.
Suppose that such a node $u'$ exists.
Let $u$ its child, let $u''$ be its parent, and
let $V$, $V'$ and $V''$ be the vertex sets associated with the nodes $u$, $u'$ and $u''$, respectively.
If $V'\cap V''=V\cap V'$, then we can modify the tree $T$ by replacing the edge $uu'$ with $uu''$, 
which increases the number of leaves and so contradicts the choice of $T$.
So, we assume that $V'\cap V''\not=V\cap V'$.

Suppose that $|V'|>4$.
If $V'\cap V''\subset V\cap V'$, let $v$ be a vertex contained in $(V\cap V')\setminus (V'\cap V'')$.
Otherwise, we have that $V'\cap V''\not\subseteq V\cap V'$ and thus $V'\cap V''\not\subseteq V$, and
we let $v$ be a vertex contained in $(V'\cap V'')\setminus V$.
Remove $v$ from $V'$ and let $G'$ be the $4$-chordal graph corresponding to this modified tree-decomposition.
Any clique transversal of $G'$ is also a clique transversal of $G$, contradictory to the choice of $G$ as a minimum counterexample.

Hence, we can assume that $|V'|=4$. Let $V=\{v_1,v_2,v_3,v_4\}$ and $V'=\{v_2,v_3,v_4,v_5\}$.
By symmetry, we can assume that $V'\cap V''\subseteq\{v_3,v_4,v_5\}$ (recall that $V'\cap V''\not=V\cap V'$ and $V'$ is a maximal clique of $G$).
We now construct another $4$-chordal graph $H$ as follows.
Let $X$ be a two-element subset of $V''$ that contains all the vertices of $V''\cap\{v_3,v_4\}$.
We construct a new tree decomposition $S$ from $T$ by removing the nodes $u$ and $u'$ and
introducing a new node $s$ associated with the set $X\cup\{w,w'\}$ and adjacent to the node $u''$,
where $w$ and $w'$ are two new vertices.
Let $H$ be the chordal graph with the tree-decomposition $S$. Observe that $H$ is $4$-chordal and
all maximal cliques of $G$, except for those induced by $V$ and $V'$, are also maximal cliques of $H$.

Let $S'$ be a nice tree-decomposition of $H$ with its root node corresponding to the clique of $H$ induced by $X\cup\{w,w'\}$.
Observe that the subtree of $S'$ rooted at the child of the root of $S'$ containing a node associated with $V''$ is a branch,
which shares the vertices of $X$ with the rest of the graph.
Hence, we can process all nodes of $S'$ except for the root, saving one \zloty{} by Lemma~\ref{lm-branch}.
Since neither of the new vertices $w$ and $w'$ is associated with a node of this branch,
neither of the vertices $w$ and $w'$, nor their \zloty{}s, are removed during this process.
In addition, the set of red vertices intersects all maximal cliques of $G$, except possibly those induced by $V$ and $V'$. 
Hence, one can think of the current state as if we had processed all nodes of $T$ except for $u$ and $u'$,
while removing all vertices of $G$ except for $v_1,\ldots,v_4$, and saving at least one \zloty{}.
If at least one of the vertices $v_3$ and $v_4$ is red, the set of red vertices is a clique transversal of $G$ and 
we save four \zloty{}s assigned to the vertices $v_1$ and $v_2$.
If none of the vertices $v_3$ or $v_4$ are red, we color one of them red and remove eight \zloty{}s assigned
to the vertices $v_1,\ldots,v_4$; in this way, we have saved an additional one \zloty{}.
In total, we have saved at least two \zloty{}s and so the constructed clique transversal of $G$
contains at most $2(|G|-1)/7$ vertices.

Hence, we can assume onwards that every node adjacent to a leaf has at least two children.
Consider a node $u$ of $T$ that has at least two children that are leaves; note that such a node $u$ must exist.
Let $u'$ and $u''$ be two such children, and
let $V$, $V'$ and $V''$ be the sets of vertices of $G$ associated with the nodes $u$, $u'$ and $u''$ respectively.
Note that $|V\cap V'|=|V\cap V''|=3$, $V'\cap V''\subseteq V\cap V'$ and $V'\cap V''\subseteq V\cap V''$.
Let $k=|V'\cap V''|$ and observe that $k\in\{0,1,2,3\}$.
Further let $v'$ be the vertex of $V'\setminus V$, $v''$ the vertex of $V''\setminus V$,
$v_1,\ldots,v_3$ the vertices of $V\cap V'$, and $v_{4-k},\ldots,v_{6-k}$ the vertices of $V\cap V''$.

We now show it must hold that $k=2$.
If $k=3$, obtain a nice tree-decomposition $T'$ from $T$ by subdividing edges of $T$ as necessary.
Note that $u$ has remained the parent of both $u'$ and $u''$.
We start processing this nice tree-decomposition
by removing the two vertices contained in $V'\setminus V$ and $V''\setminus V$, removing the nodes $u'$ and $u''$, 
and introducing a distinguished triple formed by the vertices of $V'\cap V''$.
This results in saving two \zloty{}s. We process the rest of the graph using the basic algorithm.

If $k\in\{0,1\}$, we fix a nice tree-decomposition $T'$ of $G$ such that
its root is associated with $(V\cap V')\cup(V\cap V'')$ with 
two children associated with $V'$ and $V''$.
Such a nice tree-decomposition $T'$ indeed exists: 
if $V\not=\{v_1,\ldots,v_{6-k}\}$, 
introduce a new node $w$ associated with the set $\{v_1,\ldots,v_{6-k}\}$,
make $w$ a child of $u$, and make both $u'$ and $u''$ children of $w$;
if $V=\{v_1,\ldots,v_{6-k}\}$, set $w$ to be $u$.
Next, reroot the tree-decomposition at $w$ and subdivide edges as necessary to obtain a nice tree-decomposition.
Note that neither the edge $u'w$ nor the edge $u''w$ needs subdividing.

We now process all the nodes of $T'$ except for $w$, $u'$ and $u''$ using the basic algorithm.
Note that neither of the vertices $v'$ and $v''$ can be contained in a distinguished tuple.
If any vertex among $v_1,\ldots,v_{6-k}$ is contained in either at least three distinguished triples, or
a distinguished pair and another distinguished tuple, color that vertex red and remove the \zloty{}s
assigned to it, including the distinguished tuples containing it.
Continue while such a vertex exists. Note that we remove at least seven \zloty{}s each time.
If two distinguished triples share at least two vertices,
replace them with a single distinguished pair containing two of the shared vertices, and
repeat until no such two distinguished triples exist.

We next analyze two cases depending on the value of $k$.
\begin{itemize}
\item {\bf Case $k=1$:}
  If the vertices $v_1,\ldots,v_5$ are contained in at least two distinguished tuples, then these tuples are
  either two disjoint distinguished pairs, a
  disjoint distinguished pair and triple, or
  two distinguished triples sharing a single vertex.
  In each case, it is possible to color two vertices from $v_1\ldots,v_5$ in such a way that
  at least one vertex from $v_1,\ldots,v_3$, at least one vertex form $v_3,\ldots,v_5$, and
  at least one vertex from each distinguished tuple is red.
  Since we have removed at least 16 \zloty{}s ($12$ \zloty{}s assigned among the vertices $v'$, $v''$ and at least
  four of the vertices $v_1,\ldots,v_5$ contained in the distinguished tuples, and
  at least four \zloty{}s assigned among the two distinguished tuples), we have saved at least two \zloty{}s.

  So, we can assume that there is at most one distinguished tuple.
  If this tuple contains the vertex $v_3$, we color $v_3$ red and remove all remaining vertices and \zloty{}s.
  In this way, we remove at least $10$ \zloty{}s (eight \zloty{}s assigned among the vertices $v'$, $v''$ and
  at least two of the vertices $v_1,\ldots,v_5$ contained in the distinguished tuple; and
  at least two \zloty{}s assigned to the distinguished tuple), so we have saved at least three \zloty{}s.

  If the single distinguished tuple does not contain the vertex $v_3$,
  we color one of its vertices red in such a way that if possible both sets $\{v_1,v_2,v_3\}$ and $\{v_3,v_4,v_5\}$ contain a red vertex.
  By symmetry, we can assume that we have colored $v_4$.
  If none of the vertices $v_1,\ldots,v_3$ are red, we color $v_3$ red (note that $v_5$ cannot be red in this case).
  We now remove all remaining vertices and \zloty{}s from the graph.
  If we colored only a single vertex red, we have removed at least $10$ \zloty{}s (eight \zloty{}s assigned among the vertices $v'$, $v''$ and
  at least two of the vertices $v_1,\ldots,v_5$ contained in the distinguished tuple; and
  at least two \zloty{}s assigned to the distinguished tuple), so we have saved at least three \zloty{}s.
  If we colored two vertices red, we have removed at least $16$ \zloty{}s ($14$ \zloty{}s assigned among the vertices $v'$, $v''$ and
  $v_1,\ldots,v_5$; and
  at least two \zloty{}s assigned to the distinguished tuple), so we have saved at least two \zloty{}s.

  Finally, we analyze the case where there is no distinguished tuple.
  If none of the vertices $v_1,\ldots,v_3$ are red or none of the vertices $v_3,\ldots,v_5$ are red,
  we color the vertex $v_3$ red and remove all remaining vertices and \zloty{}s from the graph.
  In this way, we remove at least $10$ \zloty{}s (those assigned among the vertices $v'$, $v''$ and
  at least three of the vertices $v_1,\ldots,v_5$), so we have saved at least three \zloty{}s.
  If at least one of the vertices $v_1,\ldots,v_3$ is red and if at least one of the vertices $v_3,\ldots,v_5$ is red,
  we color no vertex red and save at least the four \zloty{}s assigned to $v'$ and $v''$.
\item {\bf Case $k=0$:}
  If the vertices $v_1,\ldots,v_6$ are contained in at least three distinguished tuples,
  they must form three disjoint distinguished pairs. We color one vertex from each pair red in such a way that
  at least one vertex from $v_1,\ldots,v_3$ and at least one vertex from $v_4,\ldots,v_6$ are red.
  We then remove all remaining vertices and \zloty{}s.
  In this way, we have removed $25$ \zloty{}s ($16$ \zloty{}s assigned among the vertices $v'$, $v''$ and $v_1,\ldots,v_6$; and
  nine \zloty{}s assigned among the three distinguished pairs), so we have saved four \zloty{}s.

  If the vertices $v_1,\ldots,v_6$ are contained in two distinguished tuples,
  then the distinguished tuples together cover at least four vertices and
  it is possible to color two of these vertices red in such a way that
  at least one vertex from $v_1,\ldots,v_3$, at least one vertex from $v_4,\ldots,v_6$, and
  at least one vertex from each tuple are red.
  We then remove all remaining vertices and \zloty{}s.
  In this way, we have removed at least $16$ \zloty{}s (at least $12$ \zloty{}s assigned among the vertices $v'$, $v''$, and
  at least four from the vertices $v_1,\ldots,v_6$; and at least four \zloty{}s assigned among the two distinguished tuples), so
  we have saved at least two \zloty{}s.

  If the vertices $v_1,\ldots,v_6$ are contained in a single distinguished tuple,
  we color a vertex from this distinguished tuple red in such a way that if possible
  at least one vertex from $v_1,\ldots,v_3$ and at least one vertex from $v_4,\ldots,v_6$ are red.
  By symmetry we can assume that we have colored the vertex $v_4$ red.
  If none of the vertices $v_1,\ldots,v_3$ are red,
  we color red any one of them.
  We then remove all remaining vertices and \zloty{}s.
  If this procedure colors only a single vertex red, we have removed at least $10$ \zloty{}s (eight \zloty{}s assigned among the vertices $v'$, $v''$ and
  at least two of the vertices $v_1,\ldots,v_6$ contained in the distinguished tuple; and
  at least two \zloty{}s assigned to the distinguished tuple), so we have saved at least three \zloty{}s.
  If two vertices are colored red, we have removed at least $16$ \zloty{}s ($14$ \zloty{}s assigned among the vertices $v'$, $v''$ and
  $v_1,\ldots,v_6$; and
  at least two \zloty{}s assigned to the distinguished tuple), so we have saved at least two \zloty{}s.
  
  Finally, suppose that there is no distinguished tuple.
  If no vertex from $v_1,\ldots,v_3$ is red, color $v_1$ red.
  Similarly, if no vertex from $v_4,\ldots,v_6$ is red, color $v_4$ red.
  We then remove all remaining vertices and \zloty{}s.
  We have removed four \zloty{}s assigned to the vertices $v'$ and $v''$ and
  an additional six \zloty{}s for each colored vertex.
  In total, we have saved at least two \zloty{}s.
\end{itemize}
Since we save at least two \zloty{}s in both cases,
the size of the constructed clique transversal is at most $2(|G|-1)/7$ as desired.
We conclude that $k=2$.
Since the choice of $u$ was arbitrary,
we conclude that any two leaf nodes that are children of the same node of $T$
are associated with subsets of vertices of $G$ that share exactly two vertices.

We finish the proof by analyzing three cases based on the structure of $T$;
note that the node $u$ is still fixed.
\begin{itemize}
\item {\bf Case 1: The tree $T$ contains at least two different nodes such that each has at least two children that are leaves.}

Let $\hat u$ be another node of $T$ with two children $\hat u'$ and $\hat u''$ that are leaves,
let $\hat v_1,\ldots,\hat v_4$ be the four vertices associated with $\hat u'$, and
let $\hat v_3,\ldots,\hat v_6$ be the four vertices associated with $\hat u''$.
Next obtain a nice tree-decomposition from $T$ by subdividing edges of $T$ as necessary, and
observe that both $u'$ and $u''$ have remained as children of $u$ and
both $\hat u'$ and $\hat u''$ have remained as children of $\hat u$.
We start processing the graph $G$ by removing the vertices $v_1$, $v_6$, $\hat v_1$ and $\hat v_6$,
removing the nodes $u'$, $u''$, $\hat u'$ and $\hat u''$, and
introducing two distinguished pairs $\{v_3,v_4\}$ and $\{\hat v_3,\hat v_4\}$.
In this way, we save two \zloty{}s.
The rest of the graph is then processed following the basic algorithm.

\item {\bf Case 2: The first case does not apply and the root of $T$ has a non-leaf child.}
Since Case 1 does not apply, every inner node of $T$ has at most one child that is not a leaf.
Further, since we have chosen $T$ to be a tree with the maximum number of leaves among all suitable choices of $T$,
the root $r$ of $T$ has at least two children.
It follows that $r$ has two children, one is a leaf and the other is not.
In particular, $u\not=r$.
Let $r'$ be the child of $r$ that is a leaf,
$w_1,\ldots,w_4$ the vertices associated with $r'$, and
$w_5$ one of the vertices associated with $r$ but not with $r'$.
By symmetry, we can assume that the vertices $w_2$, $w_3$ and $w_4$ are associated with $r$.

If the root $r$ is associated with at least five vertices of $G$,
we introduce a new node $u_0$ associated with the vertices $w_2,\ldots,w_5$,
make both $r$ and $r'$ to be children of $u_0$, and root the tree at the node $u_0$;
if the root $r$ is associated with exactly the vertices $w_2,\ldots,w_5$, we set $u_0$ to be $r$.
Next obtain a nice tree-decomposition from $T$ by subdividing edges of $T$ as necessary.
Note that $u'$ and $u''$ have remained as children of $u$, and $r'$ has remained as a child of $u_0$.
We start processing the graph $G$ by removing the vertices $v_1$ and $v_6$,
removing the nodes $u'$ and $u''$, and
introducing a distinguished pair $\{v_3,v_4\}$.
In this way, we have saved one \zloty{}.
We then process the remaining nodes except for $u_0$ and $r'$.
If one of the vertices $w_2$, $w_3$ and $w_4$ is red,
the vertex $w_1$ is removed using Rule G5 and we save an additional two \zloty{}s.
Otherwise, $w_1$ is removed using Rule B5 and the vertices $w_2$, $w_3$ and $w_4$ become a distinguished triple.
Observe that one of the good rules applies when processing the root $u_0$ and
we save an additional one \zloty{}. In both cases, we have saved at least two \zloty{}s in total.

\item {\bf Case 3: All children of the root of $T$ are leaves.}
Observe that $u$ is the root of $T$ and $T$ is a nice tree-decomposition.
If the root has only two children, then the vertex $v_3$ is contained in all cliques, and
the statement of the theorem follows.
If the root has exactly three children, then its third child
is associated with exactly two of the vertices $v_2,v_3,v_4$;
in particular, the third child is associated with $v_3$ or $v_4$ (or both).
By symmetry, we can assume that it is associated with $v_3$.
Consequently, the vertex $v_3$ is contained in all cliques, and
the statement of the theorem again follows.

We now assume that the root has at least four children, i.e.,
it has two additional children $\hat u'$ and $\hat u''$.
Let $\hat v_1,\hat v_2,\hat v_3,\hat v_4$ be the vertices associated with $\hat u'$ and
let $\hat v_3,\hat v_4,\hat v_5,\hat v_6$ be the vertices associated with $\hat u''$ listed
in such a way that the vertices $\hat v_2,\hat v_3,\hat v_4,\hat v_5$ are associated with the root (some of
the vertices $v_2,v_3,v_4,v_5$ might be among the vertices $\hat v_2,\hat v_3,\hat v_4,\hat v_5$).
We now process the graph $G$.
We start with removing the vertices $v_1$, $v_6$, $\hat v_1$ and $\hat v_6$,
removing the nodes $u'$, $u''$, $\hat u'$ and $\hat u''$, and
introducing distinguished pairs $\{v_3,v_4\}$ and $\{\hat v_3,\hat v_4\}$.
In this way, we save two \zloty{}s. We then process the rest of the graph using the basic algorithm.
\end{itemize}
Since we save at least two \zloty{}s in each of the cases,
the size of the constructed clique transversal is at most $2(|G|-1)/7$ as desired.
\end{proof}

\section{Lower bound}

We conclude by showing that the bound given in Theorem~\ref{thm-main} is tight.
To this end, we extend the construction presented in~\cite{bib-andreae96+} for $n\mod 7=1$ to all values of $n$.

\begin{proposition}
\label{prop-lower}
For every $n\ge 5$,
there exists an $n$-vertex $4$-chordal graph with no clique transversal with fewer than $\lfloor 2(n-1)/7\rfloor$ vertices.
\end{proposition}

\begin{proof}
If $n\in\{5,6,7\}$, the expression $\lfloor 2(n-1)/7\rfloor$ is equal to one and
the complete graph $K_n$ shows that the assertion of the proposition is true.

\begin{figure}
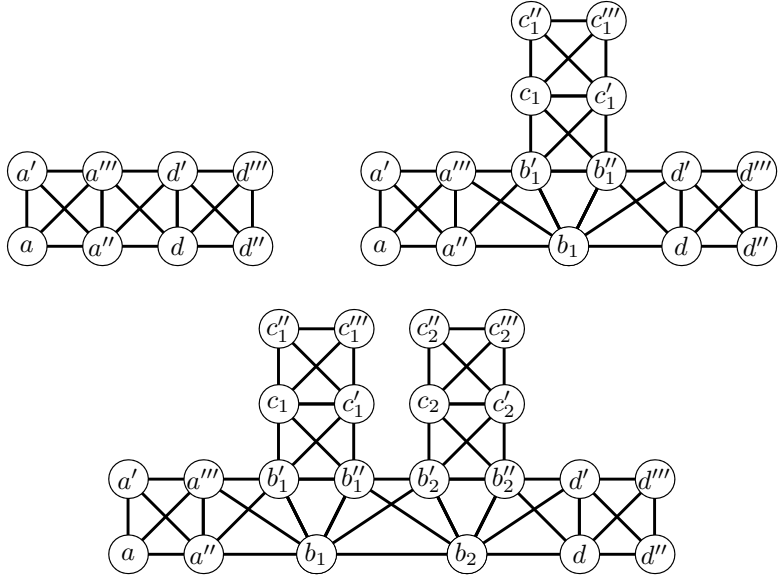

\begin{center}
\epsfbox{chord-k4.1} \hskip 10mm
\epsfbox{chord-k4.2} \vskip 5mm
\epsfbox{chord-k4.3}
\end{center}
\caption{The graphs $H_0$, $H_1$ and $H_2$.}
\label{fig-lower1}
\end{figure}

We next recall the construction presented in~\cite{bib-andreae96+}.
Let $H_k$, $k\ge 0$, be the following graph with $7k+8$ vertices (see Figure~\ref{fig-lower1}).
The vertex set of $H_k$ is formed from $7k$ vertices $b_i,b'_i,b''_i,c_i,c'_i,c''_i,c'''_i$, $i\in\{1,\ldots,k\}$, and 
eight vertices $a,a',a'',a'''$ and $d,d',d'',d'''$.
The vertices $a,a',a'',a'''$ form a $4$-clique;
each triple of vertices $b_i,b'_i,b''_i$, $i\in\{1,\ldots,k\}$, forms a $3$-clique;
each quadruple of vertices $c_i,c'_i,c''_i,c'''_i$, $i\in\{1,\ldots,k\}$, forms a $4$-clique; and
the vertices $d,d',d'',d'''$ also form a $4$-clique.
Note that these $2k+2$ cliques are vertex disjoint.
In addition, $H_k$ contains $4$-cliques formed by vertices $b'_i,b''_i,c_i,c'_i$ for $i\in\{1,\ldots,k\}$; 
vertices $b_i,b''_i,b_{i+1},b'_{i+1}$ for $i\in\{1,\ldots,k-1\}$; 
and finally vertices $a'',a''',b_1,b'_1$, and $b_k,b''_k,d,d'$.
Observe that $H_k$ is a $4$-chordal graph with $7k+8$ vertices and $2k+2$ disjoint maximal cliques.
This proves the proposition for $n\mod 7=1$.

For $n\mod 7\not=1$, we proceed as follows.
Let $k$ be the largest integer such that $7k+8\le n$ and let $z=(n-1)\mod 7$.
If $z\le 3$, we add $z$ new vertices $e_1,\ldots,e_z$ to the graph $H_k$ and
join each of them to all of the four vertices $d,d',d'',d'''$.
The resulting graph is $4$-chordal and its clique transversal number is at least $2k+2=\lfloor 2(n-1)/7\rfloor$.
If $z\ge 4$, we add $z$ new vertices $e_1,\ldots,e_z$ to the graph $H_k$ such that the $z$ vertices are mutually adjacent, 
and join the vertices $e_1$, $e_2$ to both $d''$ and $d'''$.
The resulting graph is $4$-chordal. Since the clique formed by the vertices $e_1,\ldots,e_z$
is vertex-disjoint from the $2k+2$ disjoint maximal cliques of $H_k$ identified earlier,
the clique transversal number of the resulting graph is at least $2k+3=\lfloor 2(n-1)/7\rfloor$.
\end{proof}

\section*{Acknowledgement}

The first author would like to thank the Undergraduate Research Support Scheme of the University of Warwick
for support provided to him during the summer 2015 while carrying out this work.


\begin{thebibliography}{9}
\bibitem{bib-aigner+}
M.~Aigner and T.~Andreae:
{\em Vertex-sets that meet all maximal cliques of a graph\/},
manuscript, 1986.
\bibitem{bib-andreae98}
T.~Andreae:
{\em On the clique-transversal number of chordal graphs\/},
Discrete Math. {\bf 191} (1998), 3--11.
\bibitem{bib-andreae96+}
T.~Andreae and C.~Flotow:
{\em On covering all cliques of a chordal graph\/}, 
Discrete Math. {\bf 149} (1996), 299--302.
\bibitem{bib-erdos92+}
P.~Erd\H os, T.~Gallai, and Zs.~Tuza:
{\em Covering the cliques of a graph with vertices\/},
Discrete Math. {\bf 108} (1992), 279--289.
\bibitem{bib-flotow92}
C.~Flotow:
{\em Obere Schranken f\"ur die Clique-Transversalzahl eines Graphen\/}, 
thesis, Univ.~Hamburg, 1992. 
\bibitem{bib-gavril74}
F.~Gavril:
{\em The intersection graphs of subtrees in trees are exactly the chordal graphs\/},
J.~Combin.~Theory Ser.~B {\bf 16} (1974), 47--56.
\bibitem{bib-tuza90}
Zs.~Tuza:
{\em Covering all cliques of a graph\/},
Discrete Math. {\bf 86} (1990), 117--126.
\bibitem{bib-tuza01}
Zs.~Tuza:
{\em Unsolved combinatorial problems, Part I\/},
BRICS Lecture Ser.~LS-01-1, 2001, 30 pp.
\end{thebibliography}
\end{document}